\documentclass[12pt]{amsart}

\usepackage{amssymb,amsthm,amsmath}
\usepackage[numbers,sort&compress]{natbib}
\usepackage{color}
\usepackage{graphicx}
\usepackage{tikz}
\usepackage{amssymb,amsthm,amsmath}
\usepackage[numbers,sort&compress]{natbib}
\usepackage{color}
\usepackage{graphicx}

\title[ ]{ Floquet isospectrality for  periodic  graph operators}

\author{Wencai Liu}
\address[W. Liu]{ Department of Mathematics, Texas A\&M University, College Station, TX 77843-3368, USA} \email{liuwencai1226@gmail.com; wencail@tamu.edu}

\keywords{Floquet isospectrality,  isospectrality, Fermi variety, Bloch variety, Fermi isospectrality, discrete periodic Schr\"odinger operator, separable function, triangular lattice.}

\thanks{{\em 2020 Mathematics Subject Classification.} Primary: 47B36. Secondary: 35P05,  35J10.}

\theoremstyle{plain}
\newtheorem{theorem}{Theorem}[section]

\newtheorem{corollary}[theorem]{Corollary}
\newtheorem{lemma}[theorem]{Lemma}
\newtheorem{proposition}[theorem]{Proposition}

\newcommand{\C}{\mathbb{C}}

\newcommand{\Z}{\mathbb{Z}}

\newcommand{\R}{\mathbb{R}}

\theoremstyle{plain}
\newtheorem{definition}{Definition}
\newtheorem{conjecture}{Conjecture}

\begin{document}
	
	
	\begin{abstract}
		Let  $\Gamma=q_1\mathbb{Z}\oplus q_2 \mathbb{Z}\oplus\cdots\oplus q_d\mathbb{Z}$ with arbitrary positive integers $q_l$, $l=1,2,\cdots,d$.
		Let $\Delta_{\rm discrete}+V$    be the discrete Schr\"odinger operator on $\Z^d$, 
		where $\Delta_{\rm discrete}$ is the discrete Laplacian on $\mathbb{Z}^d$ and the function $V:\mathbb{Z}^d\to \mathbb{C}$ is $\Gamma$-periodic. 
		We prove  two  rigidity theorems for   discrete periodic Schr\"odinger operators:
		\begin{enumerate}
			\item If  real-valued $\Gamma$-periodic  functions $V$ and $Y$ satisfy $\Delta_{\rm discrete}+V$ and $\Delta_{\rm discrete}+Y$ are Floquet isospectral and $Y$ is separable, then $V$ is separable.
			\item
			If  complex-valued  $\Gamma$-periodic functions $V$ and $Y$ satisfy $\Delta_{\rm discrete}+V$ and $\Delta_{\rm discrete}+Y$ are Floquet isospectral,  and both  $V=\bigoplus_{j=1}^rV_j$ and $Y=\bigoplus_{j=1}^r  Y_j$ are separable   functions,  then, up to a constant,  lower dimensional decompositions  $V_j$  and $Y_j$ are Floquet isospectral, $j=1,2,\cdots,r$.

		\end{enumerate}
		Our theorems extend the results of Kappeler.    Our approach is developed from  the author's recent work on Fermi isospectrality and can be applied to study more general lattices.

	\end{abstract}
	
	\maketitle 
	\section{Introduction and main results}

	In this paper,
	we  first study the discrete periodic Schr\"odinger equation
	\begin{equation} 
	(\Delta _{\rm discrete}u)(n)+V(n)u(n)=\lambda u(n) \label{spect_0}, n\in\Z^d,
	\end{equation}
	with the so called Floquet-Bloch boundary condition
	\begin{equation}  
	u(n+q_j\textbf{e}_j)=e^{2\pi i k_j}u(n),j=1,2,\cdots,d, \text{ and } n\in \Z^d, \label{FI}
	\end{equation}
	where $ \Delta_{\rm discrete}$ is the discrete Laplacian on $\Z^d$, $ \{\textbf{e}_j\}_{j=1}^d$ is the standard basis in $\R^d$ and $V$ is $ \Gamma$-periodic  with  $\Gamma=q_1\Z\oplus q_2 \Z\oplus\cdots\oplus q_d\Z$.

	The equation~\eqref{spect_0} with the boundary condition \eqref{FI} can be realized as  an eigen-equation of a 
	matrix $D_{V} (k)$, where  $k=(k_1,k_2,\cdots,k_d)$.
	Denote by $\sigma( D_{V} (k)) $ the (counting the algebraic multiplicity) eigenvalues of $D_{V} (k)$.
	Since we only discuss discrete   periodic Schr\"odinger equations \eqref{spect_0} and \eqref{FI} before Section 5,  for simplicity,  write $\Delta$ for $\Delta _{\rm discrete}$.
	We say  two $\Gamma$-periodic operators $\Delta+V$ and $\Delta+Y$   (or  say two  $\Gamma$-periodic  functions/potentials  $V$ and $Y$)  are
	Floquet isospectral if  
	\begin{equation}\label{gfi}
	\sigma(D_{V} (k))= \sigma(D_{Y} (k)), \text{ for any } k \in\R^d.
	\end{equation}

	The study in 	understanding when periodic potentials $V$ and $Y$ are Floquet isospectral  starts about  four decades ago
	~\cite{ERT84,MT76,kapiii,Kapi,Kapii,ERTII,gki,wa,gkii,gui90,eskin89}.  We refer readers to two  surveys  \cite{ksurvey,liujmp22} for the background  and recent development.

	The paper aims to investigate   when periodic potentials are Floquet isospectral to separable potentials.
	We say that a function $V$ on $\Z^d$ is $(d_1,d_2,\cdots,d_r)$ separable (or simply separable), where  $\sum_{j=1}^r d_j= d$ with $r\geq 2$, if there exist functions 
	$V_j$  on $\Z^{d_j}$, $j=1,2,\cdots,r$,   
	such that  for any $(n_1,n_2,\cdots,n_d)\in\Z^d$,
	\begin{align}
	V(n_1,n_2,\cdots,n_d)=&V_1(n_1,\cdots, n_{d_1})+V_2(n_{d_1+1},n_{d_1+2},\cdots,n_{d_1+d_2})\nonumber\\
	&+\cdots+V_r(n_{d_1+d_2+\cdots +d_{r-1}+1},\cdots,n_{d_1+d_2+\cdots +d_r}).\label{g61}
	\end{align}

	We   say  $V:\Z^d\to \C$ is completely  separable if $V$ is $(1,1,\cdots, 1)$ separable. 
	When there is no ambiguity, we write down \eqref{g61} as $V=\bigoplus_{j=1}^r V_j$.
	
	In   \cite{ERT84,ERTII,gki}, Eskin-Ralston-Trubowitz and Gordon-Kappeler asked the following two questions:
	\begin{enumerate}
		\item [Q1.]   If   real-valued functions $Y$ and $V$ are Floquet isospectral,  and $Y$ is  separable,  is $V$   separable?
		\item [Q2.] Assume that both $Y=\bigoplus_{j=1}^rY_j$ and $V=\bigoplus_{j=1}^r V_j$   are   separable. If $V$ and $Y$ are Floquet   isospectral,   are the lower dimensional components   $V_j$ and $Y_j$ (up to a constant) Floquet  isospectral ?
	\end{enumerate}  
	We remark that the original questions (Q1 and Q2) in \cite{ERT84,ERTII,gki} were asked in the setting of  continuous Schr\"odinger equations  with  completely separable functions. 
	It is natural to ask the above two questions in the discrete cases with arbitrary separable functions 
	(e.g., ~\cite[Section 5]{ksurvey}).  
	
	In \cite{kapiii}, Kappeler partially answered  questions Q1 and Q2:
	\begin{theorem}\cite{kapiii}\label{oldthm1}
		Let $q_1=q_2=\cdots=q_d$. 
		Then the following statements hold:
		\begin{enumerate}
			\item   If   real-valued  $\Gamma$-periodic functions $Y$ and $V$ are Floquet isospectral,  and $Y$ is completely separable,  then  $V$ completely separable. 
			\item  Assume that  complex-valued  $\Gamma$-periodic functions $Y=\bigoplus_{j=1}^dY_j$ and $V=\bigoplus_{j=1}^d V_j$   are completely separable. If $V$ and $Y$ are Floquet   isospectral,   then up to a constant, the one-dimensional functions $V_j$ and $Y_j$ are Floquet  isospectral. 
		\end{enumerate}  
	\end{theorem}

	The proof of Theorem \ref{oldthm1} relies on spectral invariants of  discrete Schr\"odinger equations. 
	We will  introduce  a  different approach to study  the  Floquet isospectrality based on the idea introduced by the author in \cite{liu2021fermi}.
	

	\begin{definition}
		The  {\it Bloch variety} $B(V)$ of  $\Delta+V$ is defined as  $$B(V)=\{(k,\lambda)\in\C^{d+1}: \det (D_{V}(k)-\lambda I)=0\}.$$
		Given $\lambda\in \C$, the Fermi surface (variety) $F_{\lambda}(V)$ is defined as the level set of the  Bloch variety:
		\begin{equation*}
		F_{\lambda}(V)=\{k\in \C^d: (k,\lambda)\in B(V)\}.
		\end{equation*}
	\end{definition}

	Both Fermi and Bloch varieties (zero sets of $\det (D_{V}(k)-\lambda I)$) play a crucial   role in the study of  
	(inverse) spectral and  related problems arising in  periodic operators,  such as embedded eigenvalue problems and integrated density of states ~\cite{GKTBook,ktcmh90,bktcm91,kv06cmp,kvcpde20,shi1,AIM16,IM14}.  Recently, the author discovered that the algebraic properties of  Fermi and Bloch varieties ($\det (D_{V}(k)-\lambda I)$ is  algebraic in appropriate coordinates) are related to more spectral problems such as the  quantum ergodicity~\cite{liu2022bloch,ms22},  spectral edges~\cite{liu1} and    isospectrality problems~\cite{liu2021fermi,liu2d}.

	A basic fact of linear algebra allows  to reformulate Floquet isospectrality as  $\det (D_{V}(k)-\lambda I)\equiv\det (D_{Y}(k)-\lambda I)$.  Like the author's recent work \cite{liu2021fermi,liu2022bloch,liu1},  we will focus on the investigation of $\det (D_{V}(k)-\lambda I)$ in this paper. One of the advantages of using  $\det (D_{V}(k)-\lambda I)$  (Bloch/Fermi varieties) is  to enable us    to apply  various  tools    from algebraic/analytic geometry  and complex analysis.

	
	

	In ~\cite{liu1}, the author   proved the irreducibility conjectures of  Bloch and Fermi varieties (see \cite{flm22,flm23,fg} for  more recent works of the  irreducibility of Bloch and Fermi varieties),   where those conjectures were only previously studied for $d=2,3$  \cite{GKTBook,batcmh92,battig1988toroidal}.  
	This implies that  functions $V$ and $Y$ are Floquet isospectral  iff   $B(V)=B(Y)$. In \cite{liu2021fermi}, 
	the author introduced a new type of inverse spectral problems-Fermi isospectrality.

	\begin{definition}\label{fermiiso}	\cite{liu2021fermi}
		Let $V$ and $Y$ be two $\Gamma$-periodic functions. We say  $V$ and $Y$ are Fermi isospectral if there exists    some $\lambda_0\in\C$ such that 	${F}_{\lambda_0} (V)={F}_{\lambda_0} (Y)$. 
	\end{definition}

	From the definitions, one can see that  Fermi isospectrality is a ``hyperplane" version (weaker assumption) of Floquet isospectrality.
	In \cite{liu2021fermi}, the author  proved the following rigidity  statements.
	
	\begin{theorem} \cite{liu2021fermi}\label{thm1}
		Assume that $d\geq 3$ and $q_j$, $j=1,2,\cdots, d$ are piecewise co-prime.
		Then the following statements hold:
		\begin{enumerate}
			\item   If   real-valued  $\Gamma$-periodic functions $Y$ and $V$ are  Fermi isospectral,  and $Y$ is   separable,  then  $V$   separable. 
			\item  Assume that   complex-valued   $\Gamma$-periodic functions    $Y=\bigoplus_{j=1}^rY_j$ and $V=\bigoplus_{j=1}^rV_j$   are  separable. If $V$ and $Y$ are Fermi  isospectral,   then up to a constant, lower dimensional decompositions V$_j$ and $Y_j$ are Floquet isospectral, $j=1,2,\cdots,r$.
		\end{enumerate}  
	\end{theorem}
	Theorem \ref{thm1} immediately implies 
	\begin{corollary} \label{oldcor}
		Assume that $d\geq 3$ and $q_j$, $j=1,2,\cdots, d$ are piecewise co-prime.
		Then the following statements hold:
		\begin{enumerate}
			\item   If   real-valued  $\Gamma$-periodic functions $Y$ and $V$ are  Floquet isospectral,  and $Y$ is   separable,  then  $V$   separable. 
			\item  Assume that  complex-valued  $\Gamma$-periodic functions  $Y=\bigoplus_{j=1}^rY_j$ and $V=\bigoplus_{j=1}^rV_j$   are  separable. If $V$ and $Y$ are Floquet isospectral,,   then up to a constant, lower dimensional decompositions V$_j$ and $Y_j$ are Floquet isospectral, $j=1,2,\cdots,r$.
		\end{enumerate}  
	\end{corollary}

	The aim of this paper is twofold. Firstly, we   answer Q1 and Q2 for arbitrary positive integers $q_l$, $l=1,2,\cdots, d$  and any types of separable functions which extends Theorem \ref{oldthm1} and Corollary \ref{oldcor}.  Secondly, we   propose a simple and new approach   to study the Floquet isospectrality based on the ideas from ~\cite{liu2021fermi}. It turns out that our approach is quite robust which works for more general lattices (see Section \ref{Stri}).
	
	Our main results in the present paper are
	\begin{theorem}\label{thmmain2}
		Assume that real-valued  $\Gamma$-periodic  functions $V$ and $Y$ are  Floquet isospectral, and $Y$ is $(d_1,d_2,\cdots,d_r)$ separable, then
		$V$ is $(d_1,d_2,\cdots,d_r)$ separable.
	\end{theorem}

	\begin{theorem}\label{thmmain3}
		
		Assume  that  complex-valued   $\Gamma$-periodic functions  $V$ and $Y$ are  Floquet isospectral,   and  both  $V=\bigoplus_{j=1}^rV_j$ and $Y=\bigoplus_{j=1}^r  Y_j$ are separable   functions.  Then, up to a constant,  lower dimensional decompositions  $V_j$  and $Y_j$ are Floquet isospectral, $j=1,2,\cdots,r$.
	\end{theorem}
	
	Theorems \ref{thmmain2} and \ref{thmmain3} imply  
	\begin{theorem}\label{mainthm}
		Assume that  real-valued functions $V$ and $Y$ are Floquet isospectral, and $Y=\bigoplus _{j=1}^r Y_j$ is    separable. Then $V=\bigoplus _{j=1}^r V_j$ is   separable.   Moreover, up to a constant,
		$Y_j$ and 	$V_j$, $j=1,2,\cdots,r$,   are Floquet isospectral.
	\end{theorem}

	Our approach is definitely developed from \cite{liu2021fermi}.  However, the technical parts are  different.  
	In the following, we will   present the main ideas  of the proof  of Theorem  \ref{thmmain2}. For simplicity, let us take  $r=2$  as an example. A basic fact of discrete Fourier transform states that a function  $V$ is $(d_1,d_2)$ separable iff for any nontrivial (non-zero modulo periodicity) $l_1\in\Z^{d_1}$ and $l_2\in\Z^{d_2}$, $\hat{V}(l_1,l_2)=0$,
	where $\hat{V}$ is the discrete Fourier transform of $V$.  In \cite{liu2021fermi}, to prove the separability of functions,   the author  used a direct way to show   $\hat{V}(l_1,l_2)=0$.
	In  the present work, we use a detour. We first show that for Floquet isospectral functions $V$ and $Y$,  $\sum_{l_1,l_2} |\hat{V} (l_1,l_2)|^2= \sum_{l_1,l_2} |\hat{Y} (l_1,l_2)|^2$, $\sum_{l_1} |\hat{V} (l_1,0)|^2= \sum_{l_1} |\hat{Y} (l_1,0)|^2$ and $\sum_{l_2} |\hat{V} (0,l_2)|^2= \sum_{l_2} |\hat{Y} (0,l_2)|^2$.  Therefore, 
	$\hat{Y}(l_1,l_2)=0$ for all nontrivial (non-zero modulo periodicity) $l_1$ and $l_2$ implies that  for all nontrivial $l_1$ and $l_2$  $\hat{V}(l_1,l_2)=0$.

	The approach in this paper is general which can be applied to study the isospectrality problems of  periodic  operators on more  general lattices.  In Section \ref{Stri},  we discuss the generalization to the triangular lattice.
	

	The rest of this paper is organized as follows.  
	In Section \ref{S2}, we  recall some basics    about the discrete  periodic Schr\"odinger equations. 
	Sections \ref{S3} and  \ref{S4}   are  devoted to proving  Theorems \ref{thmmain2} and \ref{thmmain3}. 
	In  Section \ref{Stri}, we discuss  the Floquet isospectrality of periodic operators on the triangular lattice.

	\section{ Basics }\label{S2}
	
	In this section, we first recall some basic facts about the Fermi variety, see, e.g., \cite{liu1,ksurvey,liujmp22}.
	
	Let $W$ be  a fundamental domain   for $\Gamma=q_1\Z\oplus q_2\Z\oplus \cdots \oplus q_d\Z$:
	\begin{equation*}
	W=\{n=(n_1,n_2,\cdots,n_d)\in\Z^d: 0\leq n_j\leq q_{j}-1, j=1,2,\cdots, d\}.
	\end{equation*}
	By writing out  the equation \eqref{spect_0}  on the $Q=q_1q_2\cdots q_d$ dimensional space $\{u(n),n\in W\}$, the equation \eqref{spect_0} with the boundary condition \eqref{FI} 
	translates into the eigenvalue problem for a $Q\times Q$ matrix ${D}_V(k)$. 
	
	Let $\C^{\star}=\C\backslash \{0\}$ and $z=(z_1,z_2,\cdots,z_d)$. 
	Let $z_j=e^{2\pi i k_j}, j=1,2,\cdots, d$, $\mathcal{D}_V(z)=D_V(k)$ and $\mathcal{P}_V(z,\lambda)=\det (\mathcal{D}_V(z)-\lambda I)$.

	By the basic facts of linear algebra, one has
	\begin{proposition}\label{le1}
		Let $d\geq 1$.  Two complex-valued  $\Gamma$-periodic  functions $V$ and $Y$ are Floquet isospectral  iff  
		$\mathcal{P}_V(z,\lambda) \equiv \mathcal{P}_{Y}(z,\lambda)$.
	\end{proposition}

	Define the discrete Fourier transform $\hat{V}(l) $ for $l\in {W}$ by 
	\begin{equation*}
	\hat{V}(l) =\frac{1}{{Q}}\sum_{ n\in {W} } V(n) \exp\left\{-2\pi i \left(\sum_{j=1}^d \frac{l_j n_j}{q_j} \right)\right\}.
	\end{equation*}
	For convenience, we extend $\hat{V}(l)$ to $ \Z^d$    periodically, namely, for any $l\equiv m\mod \Gamma$,
	\begin{equation*}
	\hat{V}(l)=\hat{V}(m).
	\end{equation*}

	Define
	\begin{equation}\label{gtm}
	\tilde{\mathcal{D}}_V(z)= \tilde{\mathcal{D}}_V(z_1,z_2,\cdots,z_d)= \mathcal{D}_V(z_1^{q_1},z_2^{q_2},\cdots,z_d^{q_d}),
	\end{equation}
	and 
	\begin{equation}\label{gtp}
	\tilde{\mathcal{P}}_V(z,\lambda)=\det( \tilde{\mathcal{D}}_V(z,\lambda)-\lambda I)= \mathcal{P}_V(z_1^{q_1},z_2^{q_2},\cdots,z_d^{q_d},\lambda).
	\end{equation}
	Let $$\rho^j_{n_j}=e^{2\pi  \frac{n_j}{q_j} i},$$
	where $0\leq n_j \leq q_j-1$, $j=1,2,\cdots,d$.
	
	
	A straightforward application of the  discrete Floquet transform (e.g.,  \cite{liu1,ksurvey,flm22}) leads to
	\begin{lemma}\label{lesep} 
		Let $n=(n_1,n_2,\cdots,n_d) \in {W}$ and $n^\prime=(n_1^\prime,n_2^\prime,\cdots,n_d^\prime) \in {W}$. Then 
		$\tilde{\mathcal{D}}_V(z)$ is unitarily equivalent to 		
		$
		A+B_V,
		$
		where $A$ is a diagonal matrix with entries
		\begin{equation}\label{A}
		A(n;n^\prime)=\left(\sum_{j=1}^d \left(\rho^j_{n_j}z_j+\frac{1}{\rho^j_{n_j} z_j} \right)\right) \delta_{n,n^{\prime}}
		\end{equation}
		and \begin{equation}\label{gb}
		B_V(n;n^\prime)=\hat{V} \left(n_1-n_1^\prime,n_2-n_2^\prime,\cdots, n_d-n_d^\prime\right).
		\end{equation}
		In particular,
		\begin{equation*}
		\tilde{\mathcal{P}}_V(z, \lambda) =\det(A+B_V-\lambda I).
		\end{equation*}

	\end{lemma}
	
	\section{Proof of Theorem \ref{thmmain2}}\label{S3}
	
	Assume $\sum_{j=1}^r d_j=d$ with $r\geq 2$ and $d_j\geq 1$, $j=1,2,\cdots, r$. For convenience, let $d_0=0$.
	For any $l=(l_1,l_2,\cdots, l_d)\in\Z^d$ and $j=1,2,\cdots,r$, denote by 
	\begin{equation}\label{g67}
	\tilde{l}_j=(l_{1+\sum_{m=0}^{j-1} d_m},l_{2+\sum_{m=0}^{j-1} d_m},\cdots,l_{{d_j}+\sum_{m=0}^{j-1} d_m})\in\Z^{d_j}.
	\end{equation}
	
	For $j=1,2,\cdots, r$, denote by 
	\begin{equation*}
	{W}_j= \{n\in\Z^{d_j}:  0\leq n_b\leq q_{b+\sum_{m=0}^{j-1} d_m}-1, b=1,2,\cdots, d_j\}.
	\end{equation*}
	For any $n\in\Z^{d_j}$,  $j=1,2,\cdots, r$,  denote by
	\begin{equation}\label{g67new}
	n^{j\neg}=(\overbrace {0,0,\cdots,0}^ {a_j},n_1,n_2,\cdots,n_{{d_j}},\overbrace {0,0,\cdots,0}^{b_j})\in\Z^d, 
	\end{equation}
	where $a_j=\sum_{m=0}^{j-1} d_m$ for $j=1,2,\cdots,r$ and $b_j=\sum_{m=j+1}^r d_m$  for $j=1,2,\cdots, r-1$(set $b
	_r=0$).
	\begin{lemma}\label{leSeparable1}(e.g., \cite{liu2021fermi})
		A function 
		$V$ is $(d_1,d_2,\cdots,d_r)$ separable    if and only if     for any $l\in W$    with  at least two non-zero  
		$ \tilde{l}_{j},$ $j=1,2,\cdots, r$,  $
		\hat{V}(l)=0.
		$
	\end{lemma}

	Denote by $[V]$ the average of $V$:
	\begin{equation*}
	[V]=\frac{1}{Q}\sum_{n\in W}V(n).
	\end{equation*}
	
	\begin{lemma}\label{key1}
		Assume that  real-valued  $\Gamma$-periodic functions  $V$ and $Y$ are  Floquet isospectral. Then
		\begin{equation}\label{g21}
		[V]=[Y]
		\end{equation}
		and  
		\begin{equation}\label{g55}
		\sum_{ n\in W\atop {n^\prime \in W}}
		\frac{|\hat{V}(n-n^\prime)|^2}{(-\lambda+\sum_{j=1}^d \rho^j_{n_j}z_j) (-\lambda+\sum_{j=1}^d \rho^j_{n^\prime_j}z_j) } \equiv 	\sum_{ n\in W\atop {n^\prime \in W}}
		\frac{|\hat{Y}(n-n^\prime)|^2}{(-\lambda+\sum_{j=1}^d \rho^j_{n_j}z_j) (-\lambda+\sum_{j=1}^d \rho^j_{n^\prime_j}z_j) }.
		\end{equation}
	\end{lemma}
	The proof of Lemma \ref{key1} is similar to that of Theorem 4.1~\cite{liu2021fermi}. We include a proof in the appendix for the readers' convenience.
	\begin{theorem}\label{key4}
		Assume that   real-valued  $\Gamma$-periodic functions $V$ and $ Y$ are  Floquet isospectral.    Then we have that 
		\begin{align}
		\sum_ { l\in W}| \hat{V}(l)|^2 
		=	\sum_{ l\in W} |\hat{Y}(l)|^2, \label{g4111}
		\end{align}
		and for any $j=1,2,\cdots,r$,
		\begin{align}
		\sum_{{l}\in W_j}| \hat{V}(l^{j\neg})|^2 =	\sum_{{l}\in W_j}| \hat{Y}(l^{j\neg})|^2 . \label{g411}
		\end{align}
	\end{theorem}
	
	\begin{proof}
		Letting $z=0$ in \eqref{g55}, one has \eqref{g4111}.
		Without loss of generality, we only prove \eqref{g411} for $j=1$.
		Rewrite  \eqref{g55} as
		\begin{align}
		\sum_{ \substack {n\in W, {l}\in W}}&
		\frac{|\hat{V}({l})|^2}{(-\lambda+\sum_{j=1}^d\rho^j_{n_j}z_j) (-\lambda+\sum_{j=1}^d \rho^j_{n_j+{l}_j}z_j) }\nonumber\\
		& \equiv 	\sum_{ \substack {n\in W, {l} \in W}} \frac{|\hat{Y}({l})|^2}{(-\lambda+\sum_{j=1}^d\rho^j_{n_j}z_j) (-\lambda+\sum_{j=1}^d \rho^j_{n_j+{l}_j}z_j) }.\label{g291}
		\end{align}
		Let $z_1=z_2=\cdots=z_{d_1}=0$ and consider the plane $-\lambda+\sum_{j=d_1+1}^d z_j=0$. 
		Since  for any $(n_{d_1+1},n_{d_1+2},\cdots, n_d)$  with $(n_{d_1+1},n_{d_1+2},\cdots, n_d) \neq 0$ (modulo periodicity),   planes $-\lambda+\sum_{j=d_1+1}^d z_j=0$ and $-\lambda+\sum_{j=d_1+1}^d \rho_{n_j}^jz_j=0$  are not parallel.
		This implies that the leading term   in the sum of \eqref{g291}  (double zeros in the denominators) consists  of  $\frac{1}{(-\lambda+\sum_{j=d_1+1}^{d}\rho^j_{n_j}z_j) (-\lambda+\sum_{j=d_1+1}^{d} \rho^j_{n_j+{l}_j}z_j)} $ with $n_{d_1+1}=n_{d_1+2}=\cdots=n_{d}=0$ and ${l}_{d_1+1}={l}_{d_1+2}=\cdots =l_{d} =0$. Therefore, one has
		\begin{equation}\label{g42}
		\sum_{ \substack { l\in W\\ l_m=0,d_1+1\leq m\leq d}}| \hat{V}(l)|^2 
		=\sum_{ \substack { l\in W\\ l_m=0,d_1+1\leq m\leq d}}| \hat{Y}(l)|^2 .
		\end{equation}
		This implies \eqref{g411}.
	\end{proof}

	\begin{proof}[\bf Proof of Theorem \ref{thmmain2}]
		Let $S\subset W$ 	consist  of all $l$ with   at least two of 	$ \tilde{l}_{j},$ $j=1,2,\cdots, r$  non-zero. 
		
		By Lemma \ref{leSeparable1}, one has 
		\begin{equation}\label{gn3}
		\sum_{l\in S}|\hat{Y}(l)|^2 =0.
		\end{equation}
		By Theorem \ref{key4} and  \eqref{gn3}, one has 
		\begin{equation}\label{gn4}
		\sum_{l\in S}|\hat{V}(l) |^2=0.
		\end{equation}
		By Lemma \ref{leSeparable1} again,  $V$ is separable.
		
	\end{proof}
	\section{Proof of Theorem  \ref{thmmain3} } \label{S4}
	

	\begin{proof}[\bf Proof of Theorem \ref{thmmain3}]
		By  induction,  it suffices to prove the case $r=2$.
		Recall \eqref{g21}.
		Without loss of generality, assume $[V_1]=[V_2]=[Y_1]=[Y_2]=0$ and thus  $[V]=[Y]=0$.
		Fix any $y$. Let 
		\begin{equation}\label{equ32}
		\lambda=\lambda({z}_1,z_2,\cdots, z_{d_1}, y ) =y+	\sum_{j=1}^{d_1}\left(z_j+\frac{1}{   z_j}\right).
		\end{equation}
		
		Then for any   $n\in W_1\backslash \{(0,0,\cdots, 0)\}$, 
		\begin{equation}\label{equ29}
		-\lambda+	\sum_{j=1}^{d_1}	\left(\rho_{n_j}^j z_j+\frac{1}{ \rho_{n_j}^j z_j} \right) =-y+\sum_{j=1}^{d_1}( - 1+ \rho_{-n_j}^j) \frac{1}{z_j }+\sum_{j=1}^{d_1}( - 1+ \rho_{n_j}^j) z_j.
		\end{equation}
		Comparing the coefficients of  highest degree terms (the highest degree is   $Q-|W_2|$) of $(z_1,z_2,\cdots,z_{d_1})$  in both  
		$ \tilde{\mathcal{P}}_V(z,\lambda)$  and $\tilde{\mathcal{P}}_{Y}(z,\lambda)$ with $\lambda$ given by \eqref{equ32}, one has that 
		\begin{equation*}
		\tilde{ \mathcal{P}}_{Y_2}({z}_{d_1+1},z_{d_1+2},\cdots, z_{d}, y)\equiv 	\tilde{ \mathcal{P}}_{V_2}({z}_{d_1+1},z_{d_1+2},\cdots, z_{d},y).
		\end{equation*}
		This implies 
		$ V_2$ and $Y_2$ are Floquet isospectral. Interchanging, $V_1$ and $V_2$, $Y_1$ and $Y_2$, we have 	$ V_1$ and $Y_1$ are Floquet isospectral. We complete the proof.

	\end{proof}
	\section{Floquet isospectrality for periodic Schr\"odinger operators on the triangular lattice} \label{Stri}
	
	We consider  discrete periodic Schr\"odinger equations on the triangular lattice,
	\begin{equation}
	(\Delta_{\rm Tri} u)(n)+V(n)u(n)=\lambda u(n) \label{spect_03}, n\in\Z^2,
	\end{equation}
	with the   boundary condition
	\begin{equation}  
	u(n+q_j\textbf{e}_j)=e^{2\pi i k_j}u(n),j=1,2 \text{ and } n\in \Z^2, \label{FI3}
	\end{equation}
	where $ \{\textbf{e}_1=(0,1), \textbf{e}_2=(0,1)\}$  is the standard basis in $\R^2$, $V$ is $ \Gamma$-periodic  with  $\Gamma=q_1\Z\oplus q_2\Z$,  and $\Delta_{\rm Tri}$ is the discrete Laplacian on the triangular lattice, namely
	\begin{align*}
	(\Delta_{\rm Tri} u)(n_1,n_2)=&u(n_1+1,n_2)+u(n_1-1,n_2)+u(n_1,n_2+1)+u(n_1,n_2-1)\\
	&+u(n_1+1,n_2-1)+u(n_1-1,n_2+1).
	\end{align*}
	We say   $\Delta_{\rm Tri}+V$ and $\Delta_{\rm Tri}+Y$   are
	Floquet isospectral if   for any $k\in\R^2$,  $\Delta_{\rm Tri}+V$ and $\Delta_{\rm Tri}+Y$ with the   boundary  condition \eqref{FI3} 
	have  the same eigenvalues.
	\begin{theorem}\label{thmmain2tri}
		Assume that  real-valued functions $V$ and $Y$ on $\Z^2$ satisfy that $\Delta_{\rm Tri}+V$ and $\Delta_{\rm Tri}+Y$   are  Floquet isospectral, and $Y$ is   separable, then
		$V$ is   separable.
	\end{theorem}

	%
	%

	We can  transfer all notions for discrete periodic Schr\"odinger equations on the lattice $\Z^2$ to the triangular lattice. For example, we will use $D_{{\rm Tri}, V}(k)$ to represent the equation \eqref{spect_03} with the boundary condition \eqref{FI3}.
	Similarly, we will use notations $ \mathcal{D}_{{\rm Tri},V}$,  $	\tilde{\mathcal{P}}_{{\rm Tri},V}(z, \lambda)$,
	$A_{\rm Tri}$ and $B_{{\rm Tri},V}$.
	
	By the standard discrete Floquet transform (e.g.,  \cite{flm22,ksurvey}), one has
	\begin{lemma}\label{lesep3} 
		The matrix  
		$\tilde{\mathcal{D}}_{{\rm Tri},V} (z)$ is unitarily equivalent to 		
		$
		A_{\rm Tri}+B_{{\rm Tri},V},
		$
		where $A_{\rm Tri}$ is a diagonal matrix with entries
		\begin{equation}\label{A3}
		A_{\rm Tri}(n;n^\prime)=\left(\rho^1_{n_1}z_1+\frac{1}{\rho^1_{n_1} z_1} +\rho^2_{n_2}z_2+\frac{1}{\rho^2_{n_2} z_2} +\frac{\rho^2_{n_2}z_2}{\rho^1_{n_1}z_1}+\frac{\rho^1_{n_1}z_1}{\rho^2_{n_2} z_2} \right) \delta_{n,n^{\prime}}
		\end{equation}
		and \begin{equation}\label{gb3}
		B_{{\rm Tri},V}(n;n^\prime)=\hat{V} \left(n_1-n_1^\prime,n_2-n_2^\prime\right).
		\end{equation}
		In particular,
		\begin{equation*}
		\tilde{\mathcal{P}}_{{\rm Tri},V}(z, \lambda) =\det(A_{{\rm Tri}}+B_{{\rm Tri},V}-\lambda I).
		\end{equation*}

	\end{lemma}
	
	\begin{proof} [\bf  Proof of Theorem \ref{thmmain2tri}     ]
		By Lemma \ref{lesep3}, we can  first show that \eqref{g21} and \eqref{g55} hold by a similar argument in the appendix. One can see that the proof of Theorem \ref{thmmain2}  only uses   \eqref{g21} and \eqref{g55}. So  Theorem  \ref{thmmain2tri}   holds.

	\end{proof}

	%
	%

	\appendix
	\section{Proof of Lemma \ref{key1}}
	For $a=(a_1,a_2,\cdots,a_d) \in \Z^d$ and $z^a=z_1^{a_1}z_2^{a_2}\cdots z_d^{a_d}$, denote by $|a|=a_1+a_2+\cdots+a_d$ the degree of $z^a$. 
	\begin{proof}
		By  Prop. \ref{le1}, one has
		\begin{equation}\label{equ1}
		\mathcal{P}_V(z,\lambda)\equiv\mathcal{P}_{Y}(z,\lambda),
		\end{equation}
		and	 hence
		\begin{equation}\label{equ28}
		\tilde{\mathcal{P}}_V(z,\lambda)\equiv \tilde{\mathcal{P}}_{Y}(z,\lambda).
		\end{equation}
		
		For any $n=(n_1,n_2,\cdots,n_d)\in W$, let
		\begin{equation}
		t_{n}(z,x)=x+\sum_{j=1}^d \left(\rho^j_{n_j}z_j+\frac{1}{\rho^j_{n_j}z_j} \right).
		\end{equation}
		Clearly, $t_{n}(z,[V]-\lambda)$ is the $n$-th diagonal entry of the matrix $A+B_V$.
		By Lemma \ref{lesep}, 
		direct calculations imply that 
		\begin{align}
		\tilde{\mathcal{P}}_V(z,\lambda)=&\det (A+B_V-\lambda I)\nonumber\\
		=&
		\prod_{n\in W}t_{n}(z,[V]-\lambda)-\frac{1}{2} \sum_{ n\in W, n^\prime \in W\atop{n\neq n^\prime} }
		\frac{\prod_{n\in W}t_{n}(z,[V]-\lambda) }{t_{n^\prime}(z,[V]-\lambda)t_{n}(z,[V]-\lambda)} |B_V(n;n^\prime)|^2\nonumber\\
		&+\text{ products of at most } q_1q_2\cdots q_d-3 \text{ terms of } t_{n}(z,[V]-\lambda).\label{gtmp}
		\end{align}
		Let 
		\begin{equation*}
		h(z ,\lambda)=\prod_{n\in W}\left(-\lambda+\sum_{j=1}^d \rho^j_{n_j} z_j\right),
		\end{equation*}
		and 
		$h^1_V(z,\lambda)$ be all terms in $\tilde{\mathcal{P}}_V(z,\lambda)$ consisting of $z^a\lambda^b$ with $|a|+b=q_1q_2\cdots q_d-1$.

		By \eqref{gtmp}, one has  that	$h^1_V(z,\lambda)$  must come from $\prod_{n\in W}t_{n}(z,[V]-\lambda)$ and hence  
		\begin{equation}\label{equ2}
		h^1_V(z,\lambda)=[V]\left(\sum_{n\in W} \frac{h(z,\lambda)}{-\lambda+ \sum_{j=1}^d \rho^j_{n_j}z_j}\right).
		\end{equation}
		
		By \eqref{equ28} and \eqref{equ2}, one has
		\begin{equation}\label{g1}
		\sum_{n\in W} \frac{[V]}{ -\lambda+\sum_{j=1}^d \rho^j_{n_j}z_j}\equiv \sum_{n\in W} \frac{[Y]}{ -\lambda+\sum_{j=1}^d \rho^j_{n_j}z_j}.
		\end{equation}
		This implies $[V]=[Y]$.

		Let  ${h}_V^2(z,\lambda)$ ($\bar{h}_V^2(z,\lambda)$) be  all terms in  $\tilde{\mathcal{P}}_V(z,\lambda)$ ($\prod_{n\in W}t_{n}(z,[V]-\lambda)$)
		consisting of $z^a\lambda^b$ with $|a|+b=q_1q_2\cdots q_d-2$.
		
		By the fact that  $[V]=[Y]$, one has that 
		$\bar{h}_V^2(z,\lambda)\equiv\bar{h}_Y^2(z,\lambda)$. Therefore, by~\eqref{equ28}, we have that 
		\begin{equation}\label{g22}
		h^2_V(z,\lambda) -\bar{h}_V^2(z,\lambda)\equiv	h^2_Y(z,\lambda) -\bar{h}_Y^2(z,\lambda).
		\end{equation}
		By \eqref{gtmp}, one has that
		\begin{equation}\label{equ3}
		h^2_V(z,\lambda)-	\bar{h}^2_V(z,\lambda)=\frac{h(z,\lambda)}{2} \sum_{ n\in W, n^\prime \in W\atop{n\neq n^\prime} }
		\frac{-|B_V(n;n^\prime)|^2}{(-\lambda+\sum_{j=1}^d \rho^j_{n_j}z_j) (-\lambda+\sum_{j=1}^d \rho^j_{n^\prime_j}z_j) }.
		\end{equation}
		By \eqref{g22} and \eqref{equ3}, one has
		\begin{align*}
		\sum_{ n\in W, n^\prime \in W\atop{n\neq n^\prime} }&
		\frac{|B_V(n;n^\prime)|^2}{(-\lambda+\sum_{j=1}^d \rho^j_{n_j}z_j) (-\lambda+\sum_{j=1}^d \rho^j_{n^\prime_j}z_j) } \\
		&\equiv \sum_{ n\in W, n^\prime \in W\atop{n\neq n^\prime} }
		\frac{|B_Y(n;n^\prime)|^2}{(-\lambda+\sum_{j=1}^d \rho^j_{n_j}z_j) (-\lambda+\sum_{j=1}^d \rho^j_{n^\prime_j}z_j) }  ,
		\end{align*}
		and hence
		\begin{align}
		\sum_{ n\in W, n^\prime \in W\atop{n\neq n^\prime} }&
		\frac{|\hat{V}(n-n^\prime)|^2}{(-\lambda+\sum_{j=1}^d \rho^j_{n_j}z_j) (-\lambda+\sum_{j=1}^d \rho^j_{n^\prime_j}z_j) }\nonumber\\
		& \equiv \sum_{ n\in W, n^\prime \in W\atop{n\neq n^\prime} }
		\frac{|\hat{Y}(n-n^\prime)|^2}{(-\lambda+\sum_{j=1}^d \rho^j_{n_j}z_j) (-\lambda+\sum_{j=1}^d \rho^j_{n^\prime_j}z_j) }.\label{g4}
		\end{align}
		By \eqref{g21} and \eqref{g4}, one has \eqref{g55}.
	\end{proof}

\section*{Acknowledgments}

W. Liu was 
supported by NSF  DMS-2000345 and DMS-2052572.
 



\begin{thebibliography}{10}
	
	\bibitem{AIM16}
	K.~Ando, H.~Isozaki, and H.~Morioka.
	\newblock Spectral properties of {S}chr\"{o}dinger operators on perturbed
	lattices.
	\newblock {\em Ann. Henri Poincar\'{e}}, 17(8):2103--2171, 2016.
	
	\bibitem{battig1988toroidal}
	D.~B{\"a}ttig.
	\newblock {\em A toroidal compactification of the two dimensional
		{B}loch-manifold}.
	\newblock PhD thesis, ETH Zurich, 1988.
	
	\bibitem{batcmh92}
	D.~B\"{a}ttig.
	\newblock A toroidal compactification of the {F}ermi surface for the discrete
	{S}chr\"{o}dinger operator.
	\newblock {\em Comment. Math. Helv.}, 67(1):1--16, 1992.
	
	\bibitem{bktcm91}
	D.~B\"{a}ttig, H.~Kn\"{o}rrer, and E.~Trubowitz.
	\newblock A directional compactification of the complex {F}ermi surface.
	\newblock {\em Compositio Math.}, 79(2):205--229, 1991.
	
	\bibitem{eskin89}
	G.~Eskin.
	\newblock Inverse spectral problem for the {S}chr\"{o}dinger equation with
	periodic vector potential.
	\newblock {\em Comm. Math. Phys.}, 125(2):263--300, 1989.
	
	\bibitem{ERT84}
	G.~Eskin, J.~Ralston, and E.~Trubowitz.
	\newblock On isospectral periodic potentials in {${\bf R}\sp{n}$}.
	\newblock {\em Comm. Pure Appl. Math.}, 37(5):647--676, 1984.
	
	\bibitem{ERTII}
	G.~Eskin, J.~Ralston, and E.~Trubowitz.
	\newblock On isospectral periodic potentials in {${\bf R}\sp{n}$}. {II}.
	\newblock {\em Comm. Pure Appl. Math.}, 37(6):715--753, 1984.
	
	\bibitem{fg}
	M.~Faust and J.~Lopez.
	\newblock Irreducibility of the dispersion relation for periodic graphs.
	\newblock {\em arXiv preprint arXiv:2302.11534}, 2023.
	
	\bibitem{flm22}
	J.~Fillman, W.~Liu, and R.~Matos.
	\newblock Irreducibility of the {B}loch variety for finite-range
	{S}chr\"{o}dinger operators.
	\newblock {\em J. Funct. Anal.}, 283(10):Paper No. 109670, 22, 2022.
	
	\bibitem{flm23}
	J.~Fillman, W.~Liu, and R.~Matos.
	\newblock Irreducibility of the {F}ermi variety for many-vertex models.
	\newblock {\em Preprint}, 2023.
	
	\bibitem{GKTBook}
	D.~Gieseker, H.~Kn\"{o}rrer, and E.~Trubowitz.
	\newblock {\em The geometry of algebraic {F}ermi curves}, volume~14 of {\em
		Perspectives in Mathematics}.
	\newblock Academic Press, Inc., Boston, MA, 1993.
	
	\bibitem{gki}
	C.~S. Gordon and T.~Kappeler.
	\newblock On isospectral potentials on tori.
	\newblock {\em Duke Math. J.}, 63(1):217--233, 1991.
	
	\bibitem{gkii}
	C.~S. Gordon and T.~Kappeler.
	\newblock On isospectral potentials on flat tori. {II}.
	\newblock {\em Comm. Partial Differential Equations}, 20(3-4):709--728, 1995.
	
	\bibitem{gui90}
	V.~Guillemin.
	\newblock Inverse spectral results on two-dimensional tori.
	\newblock {\em J. Amer. Math. Soc.}, 3(2):375--387, 1990.
	
	\bibitem{IM14}
	H.~Isozaki and H.~Morioka.
	\newblock A {R}ellich type theorem for discrete {S}chr\"{o}dinger operators.
	\newblock {\em Inverse Probl. Imaging}, 8(2):475--489, 2014.
	
	\bibitem{Kapi}
	T.~Kappeler.
	\newblock On isospectral periodic potentials on a discrete lattice. {I}.
	\newblock {\em Duke Math. J.}, 57(1):135--150, 1988.
	
	\bibitem{Kapii}
	T.~Kappeler.
	\newblock On isospectral potentials on a discrete lattice. {II}.
	\newblock {\em Adv. in Appl. Math.}, 9(4):428--438, 1988.
	
	\bibitem{kapiii}
	T.~Kappeler.
	\newblock Isospectral potentials on a discrete lattice. {III}.
	\newblock {\em Trans. Amer. Math. Soc.}, 314(2):815--824, 1989.
	
	\bibitem{ktcmh90}
	H.~Kn\"{o}rrer and E.~Trubowitz.
	\newblock A directional compactification of the complex {B}loch variety.
	\newblock {\em Comment. Math. Helv.}, 65(1):114--149, 1990.
	
	\bibitem{ksurvey}
	P.~Kuchment.
	\newblock An overview of periodic elliptic operators.
	\newblock {\em Bull. Amer. Math. Soc. (N.S.)}, 53(3):343--414, 2016.
	
	\bibitem{kvcpde20}
	P.~Kuchment and B.~Vainberg.
	\newblock On absence of embedded eigenvalues for {S}chr\"{o}dinger operators
	with perturbed periodic potentials.
	\newblock {\em Comm. Partial Differential Equations}, 25(9-10):1809--1826,
	2000.
	
	\bibitem{kv06cmp}
	P.~Kuchment and B.~Vainberg.
	\newblock On the structure of eigenfunctions corresponding to embedded
	eigenvalues of locally perturbed periodic graph operators.
	\newblock {\em Comm. Math. Phys.}, 268(3):673--686, 2006.
	
	\bibitem{liu2021fermi}
	W.~Liu.
	\newblock Fermi isospectrality for discrete periodic {S}chr{\"o}dinger
	operators.
	\newblock {\em Comm. Pure Appl. Math. to appear}.
	
	\bibitem{liu2d}
	W.~Liu.
	\newblock Fermi isospectrality of discrete periodic {S}chr{\"o}dinger operators
	with separable potentials on {$\mathbb{Z}^2$}.
	\newblock {\em Comm. Math. Phys. to appear}.
	
	\bibitem{liu2022bloch}
	W.~Liu.
	\newblock Bloch varieties and quantum ergodicity for periodic graph operators.
	\newblock {\em arXiv preprint arXiv:2210.10532}, 2022.
	
	\bibitem{liu1}
	W.~Liu.
	\newblock Irreducibility of the {F}ermi variety for discrete periodic
	{S}chr\"{o}dinger operators and embedded eigenvalues.
	\newblock {\em Geom. Funct. Anal.}, 32(1):1--30, 2022.
	
	\bibitem{liujmp22}
	W.~Liu.
	\newblock Topics on {F}ermi varieties of discrete periodic {S}chr\"{o}dinger
	operators.
	\newblock {\em J. Math. Phys.}, 63(2):Paper No. 023503, 13, 2022.
	
	\bibitem{MT76}
	H.~P. McKean and E.~Trubowitz.
	\newblock Hill's operator and hyperelliptic function theory in the presence of
	infinitely many branch points.
	\newblock {\em Comm. Pure Appl. Math.}, 29(2):143--226, 1976.
	
	\bibitem{ms22}
	T.~Mckenzie and M.~Sabri.
	\newblock Quantum ergodicity for periodic graphs.
	\newblock {\em arXiv preprint arXiv:2208.12685}, 2022.
	
	\bibitem{shi1}
	S.~P. Shipman.
	\newblock Eigenfunctions of unbounded support for embedded eigenvalues of
	locally perturbed periodic graph operators.
	\newblock {\em Comm. Math. Phys.}, 332(2):605--626, 2014.
	
	\bibitem{wa}
	A.~Waters.
	\newblock Isospectral periodic torii in dimension 2.
	\newblock {\em Ann. Inst. H. Poincar\'{e} Anal. Non Lin\'{e}aire},
	32(6):1173--1188, 2015.
	
\end{thebibliography}
\end{document}